\documentclass[11pt]{amsart}
\usepackage{amsmath,amssymb,amsfonts,amsthm, amscd,indentfirst}
\usepackage{amsmath,latexsym,amssymb,amsmath,
amscd,amsthm,amsxtra}
\usepackage{amsrefs}
\usepackage{hyperref}
\usepackage{url}

\usepackage{tikz} 
\usetikzlibrary{calc}
\usepackage{color}
\usepackage{pdfpages}
\usepackage{graphics}
\usepackage{enumitem}  

\newtheorem{theorem}{Theorem}[section]

\newtheorem{lemma}{Lemma}[section]
\newtheorem{corollary}{Corollary}[section]

\newtheorem{remark}{\textbf{Remark}}[section]

\def\ss{\mathbb{S}}

\def\CC{\mathbb{C}}

\def\ve{\varepsilon}

\def\lan{\langle}
\def\ran{\rangle}

\def\n{\nabla}

\def\vp{\varphi}
\def\R{{\mathbb R}}

\newtheorem{proposition}[theorem]{Proposition}

\numberwithin{equation} {section}

\begin{document}
\title[Hopf theorems for curvature flows]{Hopf type theorems for self-similar solutions of curvature flows in $\R^3$}

\begin{abstract}
In this paper we prove rigidity results for two-dimensional, closed, immersed, non-necessarily convex, self-similar solutions of a wide class of fully non-linear parabolic flows in $\R^3.$ We show this self-similar solutions are the round spheres centered at the origin provided it has genus zero and satisfies a suitable upper pinching estimate for the Gaussian curvature. As applications, we obtain rigidity results for the round sphere as the only closed, immersed, genus zero, self-similar solution of several well known flows, as the flow of the powers of mean curvature, the harmonic mean curvature flow and the $\alpha$-Gaussian curvature flow for $\alpha\in(0,1/4).$ We remark that our result does not assume any embeddedness condition.
\end{abstract}

\medskip



\author{Hil\'ario Alencar \and Greg\'orio Silva Neto \and Detang Zhou}

\thanks{Hil\'ario Alencar, Greg\'orio Silva Neto and Detang Zhou were partially supported by the National Council for Scientific and Technological Development - CNPq of Brazil.}

\date{September 21, 2020}

\address{Instituto de Matem\' atica, Universidade Federal de Alagoas, Macei\'o, AL, 57072-900, Brazil}

\email{hilario@mat.ufal.br}

\address{Instituto de Matem\'atica, Universidade Federal de Alagoas, Macei\'o, AL, 57072-900, Brazil}

\email{gregorio@im.ufal.br}

\address{Instituto de Matem\'atica e Estat\'istica, Universidade Federal Fluminense, Niter\'oi, RJ 24020, Brazil}

\email{uffzhou@gmail.com}



\maketitle
\section{Introduction}
The evolution of an embedded hypersurface  in Euclidean space by a specified function of the eigenvalues of its second fundamental form has attracted mathematicians in many areas. In particular, there are interesting results for flows by homogeneous symmetric functions of the principal curvatures, such as mean curvature,  harmonic mean curvature, inverse mean curvature and Gauss curvature flows  by  many authors in the last three decades, see \cite{Gerhardt-1}, \cite{Urbas-0}, \cite{Urbas-01}, \cite{Andrews-FN-0}, \cite{Andrews-FN-01}, \cite{Andrews-FN-1}, \cite{Andrews-FN-2}, \cite{McCoy} and references therein.

All these flows are particular cases of the more general fully non-linear flow
\begin{equation}\label{F-1}
\begin{cases}
\dfrac{\partial X}{\partial t} = W(k_1,k_2)N,\\
X(\cdot,0)=X_0,\\
\end{cases}
\end{equation}
where $k_1$ and $k_2$ are the principal curvatures of the immersions $X$, $N$ is their unitary normal vector fields, and $W\in C^1(\R^2)$. It is known this flow will be parabolic if and only if
\begin{equation}\label{pl}
\frac{\partial W}{\partial k_1}\cdot \frac{\partial W}{\partial k_2} >0.
\end{equation}

We will use a  particular combination of $k_1$ and $k_2$  which will become apparent later. A flow is said to be contracting if $W\ge 0,$ and expanding if $W\le 0$. A stationary solution of a classical curvature flow is a W-surface (Weingarten surface), satisfying the equation $W=0$. While nonround Weingarten spheres exist (for example, Hopf spheres, see \cite{Hopf}), many results state that the only Weingarten spheres satisfying some particular relation are round (cf Chern \cite{chern}). We would like to mention that  Rosenberg and Sa Earp \cite{RS} obtained examples of special W-surfaces near any Wente torus or any Kapouleas examples of surfaces with constant mean curvature.

In this paper, we will address to immersions $X:\Sigma\to\R^3$ which are solutions of the second order partial differential equation
\begin{equation}\label{soliton}
W(k_1,k_2)=-\lambda\lan X,N\ran, \ \lambda\in\R,
\end{equation}
where $k_1$ and $k_2$ are the principal curvatures of the immersion, $N$ is its unitary normal vector field, and $W\in C^1(\R^2)$. 

If $W$ is a homogeneous function of degree $\beta>0,$ i.e., $W(ak_1,ak_2)=a^\beta W(k_1,k_2),\ a>0,$ then the solutions of (\ref{soliton}) are the self-similar solutions of the flow (\ref{F-1}). Equation (\ref{soliton}) generalizes the mean curvature self-shrinker equation and will also describe the self-similar solutions of a large class of curvature flows, as the Gaussian curvature flow, the harmonic mean curvature flow, among others.

Changing the variables $x_1=k_1+k_2$ and $x_2=(k_1-k_2)^2$ we can write
\begin{equation}\label{W-Psi0}
W(k_1,k_2)=\Psi(x_1,x_2)=\Psi(k_1+k_2,(k_1-k_2)^2)=\Psi(H,H^2-4K),
\end{equation}
where $K=k_1k_2$ is the Gaussian curvature of the immersion $X$ and $H=k_1+k_2$ is its mean curvature. Therefore, the equation (\ref{soliton}) becomes
\begin{equation*}
\Psi(H,H^2-4K)=-\lambda\lan X,N\ran, \ \lambda\in\R.
\end{equation*}

The main result of this paper is the following

\begin{theorem}\label{thm-soliton-0}
Let $X:\Sigma\to\R^3$ be a closed, immersed surface of genus zero satisfying
\begin{equation}\label{soliton-0}
\Psi(H,H^2-4K)=-\lambda\lan X,N\ran,\ \lambda\in\R,
\end{equation}
where $\Psi: \mathbb{R}\times [0,+\infty)\to \mathbb{R}$ is a $C^1$ function satisfying $\frac{\partial\Psi}{\partial x_1}\neq 0$. If  there exists a real number $\ve>0$ such that
\begin{equation}\label{hyp-soliton-0}
K \leq \frac{1}{4}\left[1-\varepsilon\lambda^2\left(\|X\|^2-\lan X,N\ran^2\right)\right]H^2,
\end{equation}
then $X(\Sigma)$ is a round sphere centered at the origin and radius satisfying the equation 
\[
\lambda R = \Psi\left(\frac{2}{R},0\right).
\]
\end{theorem}
A number of remarks are listed here.
\begin{remark}
{\normalfont 
The hypothesis (\ref{hyp-soliton-0}) of Theorem \ref{thm-soliton-0} is necessary. In fact, in the Section \ref{example}, we prove that, if there exists non-spherical genus zero rotational surface which is the solution of (\ref{soliton-0}), then (\ref{hyp-soliton-0}) does not hold.
}
\end{remark}

\begin{remark}
{\normalfont
Notice that the equation $\lambda R = \Psi\left(\frac{2}{R},0\right)$ may have no solution for some $\Psi$ and $\lambda.$ In this case of non existence of solutions, Theorem \ref{thm-soliton-0} is considered as a non-existence result.
}
\end{remark}

\begin{remark}
{\normalfont  Recall that a W-surface is called special if the relation between the principal curvatures can be written in the form $\Psi(H,H^2-4K)=0$ with $\frac{\partial \Psi}{\partial x_1}\neq 0$ at every umbilic point. The Chern-Hartman-Winter's theorem (see \cite{H-W} and \cite{chern}) says that that if a closed orientable surface $\Sigma$ of genus zero is a special W-surface of class $C^2$, then $\Sigma$ is a sphere. This improves a theorem of H. Hopf, by removing the analyticity assumptions. Theorem \ref{thm-soliton-0} recovers the Chern-Hartman-Winter's theorem for special Weingarten surfaces by taking $\lambda=0$.
}
\end{remark}

The class of surfaces satisfying (\ref{hyp-soliton-0}) contains all the closed surfaces under the pinching condition
\[
K\leq \frac{1}{4}(1-\delta)H^2, \ \forall \delta>0.
\]
In fact, for given $\delta>0,$ choose $\ve\leq \frac{\delta}{\lambda^2\sup_{\Sigma}(\|X\|^2-\lan X,N\ran^2)}.$ Notice that the supremum exists since $\Sigma$ is assumed to be closed. Since the sphere is not contained in this class we obtain the following non existence result.
\begin{corollary}
There is no closed, immersed solution $X:\Sigma\rightarrow\R^3$ of (\ref{soliton-0}) with genus zero, such that $\frac{\partial\Psi}{\partial x_1}\neq 0,$ and 
\begin{equation}\label{gauss-delta}
K\leq \frac{1}{4}(1-\delta)H^2,
\end{equation}
for every $\delta>0.$
\end{corollary}

\begin{remark}
{\normalfont 
We can compare Theorem \ref{thm-soliton-0} with the results of obtained by McCoy, see \cite{McCoy}, Theorem 1.3, p. 320, and Guilfoyle and Klingenberg, see \cite{G-K}, Theorem 4.7, p. 353. The result of \cite{McCoy} holds for higher dimensions and for non-convex surfaces and  requires  more restrictions than Theorem \ref{thm-soliton-0}.  The result of \cite{G-K} requires that the surface is convex, among other conditions over the derivatives of $\Psi.$
}
\end{remark}

\begin{remark}
{\normalfont
The function $W(k_1,k_2)$ is homogeneous of degree $\beta\in\R,$ if and only if the function $\Psi$ satisfies 
\begin{equation}\label{psi-homo}
\Psi(a x_1,a^2 x_2)=a^\beta\Psi(x_1,x_2),\ a>0.
\end{equation} 
By an abuse of notation, we will call $\Psi$ a homogeneous function of degree $\beta\in\R$ if $\Psi$ satisfies (\ref{psi-homo}).

If $\Psi$ is homogeneous of degree $\beta\neq -1$ with $\Psi(1,0)>0,$ and $\lambda>0,$ then the radius of the sphere of Theorem \ref{thm-soliton-0} is given by
\[
R=\left[\lambda^{-1}2^\beta\Psi(1,0)\right]^{\frac{1}{\beta+1}}.
\]
}
\end{remark}

\begin{remark}
{\normalfont 
In \cite{G-L-M}, Theorem 1.12, p .3, Gao, Li and Ma proved that a strictly convex surface satisfying the equation $F(k_1,\ldots,k_n)+C=-\lambda\lan X,N\ran$ is a sphere, where $k_1,\ldots,k_n$ are the principal curvatures of surface, provided $C\leq 0$ and $F$ is a degree $\beta$ homogeneous function which satisfies $F>0,\ \frac{\partial F}{\partial k_i}>0$ and 
\[
\begin{aligned}
& \left(\frac{\partial F}{\partial k_i}k_j - \frac{\partial F}{\partial k_j}k_i\right)(k_i-k_j)^{-1}>0, \\
&\sum_{i=1}^n\frac{1}{k_i}\frac{\partial\log F}{\partial k_i}y_i^2 + \sum_{i,j=1}^n\frac{\partial^2\log F}{\partial k_i \partial k_j} y_iy_j>0,
\end{aligned}
\] 
for every, $i,j=1,\ldots,n$ and $(y_1,\ldots,y_n)\in\R^n.$

We observe that Theorem \ref{thm-soliton-0} holds in the same situation in dimension $2,$ i.e., for $W(k_1,k_2)=F(k_1,k_2)+C$ where $F$ is a degree $\beta$ homogeneous function and $C\in\R,$  only imposing that 
\[
\frac{\partial F}{\partial k_1}+\frac{\partial F}{\partial k_2}>0.
\] 
In this case, the radius $R$ of the sphere\ satisfies 
\[
\lambda R^{\beta+1}-CR^\beta-F(1,1)=0.
\]
}
\end{remark}
	
\begin{remark}
{\normalfont
The flow (\ref{F-1}) is a (weakly) parabolic equation if and only if
\[
\frac{\partial W}{\partial k_1}\cdot\frac{\partial W}{\partial k_2}>0\ (\geq 0),
\]
or equivalently
\[
\left(\frac{\partial\Psi}{\partial x_1}\right)^2 - 4x_2\left(\frac{\partial\Psi}{\partial x_2}\right)^2>0\ (\geq 0).
\]
Notice the hypothesis $\frac{\partial\Psi}{\partial x_1}\neq 0$ of Theorem \ref{thm-soliton-0} assures the parabolicity of the flow near the umbilical points ($x_2=0$), but this result holds even when the flow is not parabolic.
}
\end{remark}

Now we derive the consequences of Theorem \ref{thm-soliton-0} for the diverse kind of flows. The simplest situation is the mean curvature flow, when $\Psi(x_1,x_2)=x_1.$ This flow is parabolic for every value of the mean curvature $H$ and it is, by far, the most studied among all parabolic flows. The next result is a particular case of the result proven by the authors in \cite{ANZ}.

\begin{corollary}\label{H-ss}
Let $X:\Sigma\to\R^3$ be a closed, immersed self-shrinker of genus zero. If there exists a real number $\ve>0$ such that  
\begin{equation}\label{eq-Hss}
K\leq \frac14\left[1-\varepsilon\lambda^2(\|X\|^2-\lan X,N\ran^2)\right]H^2,
\end{equation}
then $X(\Sigma)$ is a round sphere of radius $\sqrt{2/\lambda}$ and center at the origin.
\end{corollary}

\begin{remark}\label{Rem-a}
{\normalfont
The hypothesis (\ref{eq-Hss}) of Corollary \ref{H-ss} is necessary. In fact, Drugan, see \cite{D}, gave an example of an immersed rotational self-shrinker, homeomorphic to the sphere, which is not the round sphere. In section \ref{example} we prove that this example of self-shrinker does not satisfy (\ref{eq-Hss}).
}
\end{remark}


Our second consequence is for the $H^\beta$-curvature flow, $\beta\in \R\backslash\{0,1\},$
\[
\frac{\partial X}{\partial t}=H^\beta N.
\]
This flow is parabolic for $H>0.$  Schulze, see \cite{Scz-1}, proved that closed (weakly) convex hypersurfaces of $\R^{n+1}$ converges to a point if $\beta\in(0,1)$ ($\beta\geq 1$) and Schn\"urer, see \cite{SA2}, and Schulze, see \cite{Scz-2}, proved that closed convex surfaces of $\R^3$ converges to a round point for $1\leq \beta \leq 5.$ For general speeds of higher homogeneity, Andrews, see \cite{Andrews-FN-2}, proved that the flow of a convex surfaces converges to a round point provided it satisfies an initial pinching condition. 

The self-similar solutions of the $H^\beta$-curvature flow satisfy the equation
\[
H^\beta = -\lambda\lan X,N\ran,\ \lambda\in\R.
\]
Our result characterizes the sphere as the only mean convex (i.e., $H\neq 0$), genus zero, closed self-similar solution of the $H^\beta$-curvature flow under an upper pinching curvature condition. Notice that the mean convex assumption is  weaker than convexity, since mean convexity admits immersed surfaces and surfaces with $K\leq 0.$
 
\begin{corollary}\label{Hn-alpha}
Let $X:\Sigma\to\R^3$ be a closed, immersed, mean convex surface of genus zero satisfying
\begin{equation}\label{soliton-Hn}
H^\beta=-\lambda\lan X,N\ran,
\end{equation}
where $\lambda>0$ and $\beta\in \R\backslash\{-1,0\}.$ If there exists a real number $\ve>0$ such that  
\begin{equation}\label{eq11}
K\leq \frac14\left[1-\varepsilon\lambda^2(\|X\|^2-\lan X,N\ran^2)\right]H^2,
\end{equation}
then $X(\Sigma)$ is a round sphere of radius $(2^\beta\lambda^{-1})^{\frac{1}{\beta+1}}$ and center at the origin.
\end{corollary}
\begin{remark}
{\normalfont
If $\beta:=\frac{m}{2n-1}\in(0,1),\ n,m\in\mathbb{N},$ then the hypothesis of mean convexity in Corollary \ref{Hn-alpha} is not necessary. Notice that in this case the flow is only weakly parabolic, becoming degenerate for the points when $H=0.$
}
\end{remark}

The next application of Theorem \ref{thm-soliton-0} is for the $\alpha$-harmonic mean curvature flow
\[
\frac{\partial X}{\partial t} =\left(\frac{K}{H}\right)^\alpha N,
\]
whose self-similar solitons satisfy the equation
\[
\left(\frac{K}{H}\right)^\alpha=-\lambda\lan X,N\ran.
\]
If $\alpha\in (0,\infty),$ then this flow is (weakly) parabolic for (weakly) convex surfaces, being degenerate for the points where $K=0.$ If we consider values of $\alpha$ such that $K$ can assume negative values, as $\alpha=\frac{m}{2n-1},\ m,n\in\mathbb{N},$ including the classical case of $\alpha=1,$ then the flow is weakly parabolic for every surface, being degenerate for the points where $K=0$ and singular for the points where $H=0.$

For $\alpha=1,$ the existence of solutions for convex surfaces initial data was proved by Andrews, see \cite{Andrews-FN-0}, who also showed that convex surfaces flowing by the harmonic mean curvature converges to a round point in finite time (in fact, the result of Andrews holds for a more wide class of degree one homogeneous functions $W$). Dieter, see \cite{Dieter}, studied the convergence of the flow for the degenerate case $K\geq 0$ and $H>0,$ Caputo and Daskalopoulos, see \cite{D-C}, and Daskalopoulos and Sesum, see \cite{D-S}, studied the highly degenerate case, where $K$ and $H$ can be simultaneously zero. The case when $K<0$ and $H<0$ was studied by Daskalopoulos and Hamilton, see \cite{D-H-2}. 

For $\alpha\in(0,1),$ Anada, see \cite{Anada}, proved the existence of non-round convex self-similar solutions of the $\alpha$-harmonic mean curvature flow. After this findings, in a joint work with Tsutsumi, see \cite{A-T}, he also investigated sufficient conditions for the $\alpha$-mean curvature flows converge to a round point. 

Our result gives conditions for a closed, mean convex, self-similar solution with genus zero of the $\frac{m}{2n-1}$-harmonic mean curvature flow to be a sphere. We remark here that the powers $\frac{m}{2n-1}, m,n\in\mathbb{N},$ allows us to work with surfaces such that $K<0$ at some points, but our technique holds for every $\alpha\in(0,1],$ if we assume that $\Sigma$ is weakly convex.


\begin{corollary}\label{harm-H}
Let $X:\Sigma\to\R^3$ be a closed, immersed, mean convex surface of genus zero satisfying
\begin{equation}\label{soliton-HK}
\left(\frac{K}{H}\right)^{\frac{m}{2n-1}}=-\lambda\lan X,N\ran,
\end{equation}
where $m, n\in\mathbb{N}$, $\lambda>0,$ and $\frac{m}{2n-1}\leq 1.$ If there exists a real number $\ve>0$ such that  
\begin{equation}\label{eq10}
K\leq \frac14\left[1-\varepsilon\lambda^2(\|X\|^2-\lan X,N\ran^2)\right]H^2,
\end{equation}
then $X(\Sigma)$ is a round sphere of radius $(2^{\frac{m}{2n-1}}\lambda)^{-\frac{2n-1}{m-2n+1}},$ centered at the origin, if $(m,n)\neq(1,1),$ and for any radius $R>0,$ centered at the origin, with $\lambda=\frac{1}{2},$ if $(m,n)=(1,1).$
\end{corollary}


The last classical flow we will discuss here and obtain consequences of Theorem \ref{thm-soliton-0} is the $\alpha$-Gaussian curvature flow 
\[
\frac{\partial X}{\partial t} = K^\alpha N,
\]
whose self-similar solutions satisfy the equation 
\[
K^\alpha = -\lambda\lan X,N\ran.
\]
This flow if (weakly) parabolic if $K>0$ ($K\geq 0$) being degenerate for the points where $K=0.$

When $\alpha=1,$ this flow is called Gaussian curvature flow, and was first introduced by Firey in 1974, see \cite{Firey}, as a model of the wearing process of convex rolling stones on a beach. He proved also that closed convex surfaces under this flows converges to a round point when they are symmetric about the origin. Tso, see \cite{Tso}, for $\alpha=1,$ and Chow, see \cite{Chow-1}, for $\alpha=1/n,$ proved the convergence to a point of a closed convex hypersurfaces of $\R^{n+1}$ under the flow. Andrews, see \cite{Andrews-1996}, proved that, for $\alpha=1/(n+2),$ convex hypersurfaces varying under the flow converges to an ellipsoid. We observe that Calabi, see \cite{Calabi}, early proved that the ellipsoids are the only hypersurfaces satisfying the equation of the self-similar solutions of the $\frac{1}{n+2}$-Gaussian curvature flow. The works of Andrews, see \cite{Andrews-2000}, \cite{AGN}, and Guan and Ni, see \cite{Guan-Ni}, proved that the flow converges to a self-similar solution for every $\alpha\geq 1/(n+2).$ To conclude the analysis of the case when $\alpha\geq 1/(n+2),$ Brendle, Choi and Daskalopoulos, see \cite{B-C-D}, proved that the only self-similar solution of the $\alpha$-Gaussian curvature flow for $\alpha>1/(n+2)$ are the round spheres. In his turn, if $\alpha<0,$ then Gerhardt, see \cite{Gerhardt-2}, proved that the only convex self-similar solution of the $\alpha$-Gaussian curvature flow is a round sphere. Moreover, he proved that the flow converges to a sphere after rescaling. 

On the other hand, Andrews, see \cite{Andrews-2000}, proved the existence of non-spherical convex self-similar solutions of the $\alpha$-Gaussian curvature flow for small $\alpha>0.$ In particular, in dimension $2,$ for $\alpha\in (0,1/10).$ This shows that if we want to characterize the sphere as the only self-similar solution of the $\alpha$-Gauss curvature flow for small values of $\alpha>0,$ then we will need some additional assumption. 

Our result provides sufficient conditions for a self-similar solution of the $\alpha$-Gaussian curvature flow, $\alpha\in(0,1/4),$ to be a round sphere. 

\begin{corollary}\label{K-alpha}
Let $X:\Sigma\to\R^3$ be a closed, convex surface satisfying
\begin{equation}\label{soliton-K}
K^\alpha=-\lambda\lan X,N\ran, \ \alpha\in(0,1/4),\ \lambda>0.
\end{equation}
If there exists a real number $\ve>0$ such that  
\begin{equation}\label{eq9}
K\leq \frac14\left[1-\varepsilon\lambda^2(\|X\|^2-\lan X,N\ran^2)\right]H^2,
\end{equation}
then $X(\Sigma)$ is a round sphere of radius $\lambda^{-\frac{1}{2\alpha+1}}$ and center at the origin.
\end{corollary}


\begin{remark}
{\normalfont
Since there are examples of closed convex self-similar solutions of the $\alpha$-Gaussian curvature flow for $\alpha\in(0,1/10),$ given by Andrews, see \cite{Andrews-2000}, at least in this cases some additional hypothesis like (\ref{eq9}) is necessary to obtain the conclusions of Corollary \ref{K-alpha}.
}
\end{remark}

\begin{remark}
{\normalfont
Corollary \ref{K-alpha} holds in a more general setting: if we choose values of $\alpha$ which allows negative values of $K,$ as for example $\alpha=\frac{m}{2n-1}\leq 1,$ $m,n\in\mathbb{N},$ then we can assume only that $\Sigma$ is a closed mean convex surface with genus zero to obtain the same conclusion, despite the flow is not parabolic in this case.
}
\end{remark}

In order to illustrate the scope of situations to which the Theorem \ref{thm-soliton-0} can be applied in the context of the curvature flows, we give here a list of examples of homogeneous functions $W(k_1,k_2)$ such that the flow (\ref{F-1}) is parabolic including negative values of $K.$

\begin{itemize}
\item[(i)] $W(k_1,k_2)=aH^2+bK,$ $a,b\in\R.$ The flow is parabolic for \[K>-\dfrac{2a(2a+b)}{b^2}H^2.\] In this case, $\frac{\partial\Psi}{\partial x_1}\neq 0$ if and only if $H\neq 0,$ i.e., the surface is mean convex. 

As a particular situation, we have $W(k_1,k_2)=|A|^2=k_1^2+k_2^2,$ by taking $a=1$ and $b=-2.$ In this case, the flow is parabolic for $K>0,$ i.e., for convex surfaces. This flow was studied by Schn\"urer in \cite{SA2}.

\item[(ii)] $W(k_1,k_2)=aH^{2\alpha}+bK^\alpha,$ $a,b>0,$ $\alpha=\frac{m}{2n-1}\geq1, m,n\in\mathbb{N}.$ The flow is parabolic for \[4a^2+2ab\left(\dfrac{K}{H^2}\right)^{\alpha-1}+b^2\left(\dfrac{K}{H^2}\right)^{2\alpha-1}>0.\] In this case, $\frac{\partial\Psi}{\partial x_1}\neq 0$ if and only if $H\neq 0,$ i.e., the surface is mean convex.

\item[(iii)] $W(k_1,k_2)=H^{\frac23}+bK^{\frac13},$ $b\in(0,2^{5/3}).$ The flow is parabolic for $K\neq 0$ and $H\neq 0.$ In this case, $\frac{\partial\Psi}{\partial x_1}\neq 0$ everywhere and it is singular for $K=0$ and $H=0.$
\end{itemize}

\section{Proofs of the main Theorem}

In the study of the mean curvature flow and the $\alpha$-Gaussian curvature flow its was proved that, under some convexity conditions, the solutions of the flow, when suitably normalized, converge to a self-similar solution. A solution $X(\cdot,t)$ of (\ref{F-1}) is called self-similar if it is an homothety of $X_0,$ i.e, if there exists an smooth positive function $\phi$ such that $\phi(0)=1$ and 
\begin{equation}\label{SS}
X(\cdot,t)=\phi(t)X_0.
\end{equation} 
For this reason, a self-similar solution of (\ref{F-1}) is also called a homothetic solution. If $W(k_1,k_2)$ is a homogeneous function of degree $\beta\in\R,$ i.e.,
\[
W(a k_1,a k_2)=a^\beta W(k_1,k_2),\ a>0,
\]
and $X(\cdot,t)=\phi(t)X_0$ is a self-similar solution, then
\[
\dfrac{\partial}{\partial t}[X(p,t)]= W(k_1(p,t),k_2(p,t))N(p,t),
\]
i.e.,
\[
\dfrac{\partial}{\partial t}[\phi(t)X_0]= W(\phi(t)^{-1}k_1(p,0),\phi(t)^{-1}k_2(p,0))N(p,0).
\]
This gives
\[
\phi'(t)X_0 = \phi(t)^{-\beta}W(k_1(p,0),k_2(p,0))N(p,0),
\]
which is equivalent to
\[
\phi'(t)\phi(t)^\beta \lan X_0(p),N(p,0)\ran = W(k_1(p,0),k_2(p,0)).
\]
This implies that $\phi'(t)\phi(t)^\beta$ is constant, i.e., there exists $\lambda\in\R$ such that
\begin{equation}\label{soliton-p0}
W(k_1(p,0),k_2(p,0)) = -\lambda \lan X_0(p),N(p,0)\ran.
\end{equation}
Notice that the sign of $-\lambda$ is the same of $\phi'(t).$ Thus the solution expands for $\lambda<0,$ and shrinks for $\lambda>0.$ The solution is stationary if $\lambda=0.$ In this case, we have the so called Weingarten surfaces, i.e., surfaces satisfying the equation
\[
W(k_1,k_2)=0,
\]
see, for example, \cite{Hopf} for more details. 

In this paper, we will work with immersions $X:\Sigma\to\R^3$ which are solutions of the second order partial differential equation
\begin{equation}\label{soliton-1111}
W(k_1,k_2)=-\lambda\lan X,N\ran, \ \lambda\in\R,
\end{equation}
where $X$ is the position vector and $N$ is the unitary normal vector field of the immersion $X$. Changing the variables
\[
x_1=k_1+k_2 \ \mbox{and}\ x_2=(k_1-k_2)^2,
\]
we have, 
\begin{equation}\label{W-Psi}
W(k_1,k_2)=\Psi(x_1,x_2)=\Psi(k_1+k_2,(k_1-k_2)^2)=\Psi(H,H^2-4K),
\end{equation}
$K=k_1k_2$ is the Gaussian curvature of $X,$ $H=k_1+k_2$ is its mean curvature, and $\Psi(x_1,x_2)$ is a real function of two real variables. We remark that, in general, we will have at least two different expressions for $\Psi$ in (\ref{W-Psi}), one for each half-space determined by the diagonal $k_1=k_2,$ except when $W$ is symmetric, i.e., $W(k_1,k_2)=W(k_2,k_1).$ In this case $\Psi$ can be defined by a unique expression. 

Notice that
\[
\frac{\partial W}{\partial k_1} = \frac{\partial \Psi}{\partial x_1} + 2(k_1-k_2)\frac{\partial \Psi}{\partial x_2} \ \mbox{and}\ \frac{\partial W}{\partial k_2} = \frac{\partial \Psi}{\partial x_1} - 2(k_1-k_2)\frac{\partial \Psi}{\partial x_2}.
\]
It implies that
\[
\frac{\partial W}{\partial k_1}\cdot\frac{\partial W}{\partial k_2} = \left(\frac{\partial \Psi}{\partial x_1}\right)^2 - 4(k_1-k_2)^2\left(\frac{\partial \Psi}{\partial x_2}\right)^2.
\]
Thus, if $\Psi$ satisfies
\begin{equation}\label{parabol-ellip-2}
\left(\frac{\partial \Psi}{\partial x_1}\right)^2 - 4x_2\left(\frac{\partial \Psi}{\partial x_2}\right)^2>0\ (\mbox{resp.} \geq 0),
\end{equation}
then the equation
\begin{equation}\label{F-1R}
\dfrac{\partial X}{\partial t}(p,t) = \Psi(H(p,t),H(p,t)^2 - 4K(p,t))N(p,t)
\end{equation}
is parabolic (resp. weakly parabolic). In this case, equation (\ref{soliton}) becomes
\begin{equation}\label{soliton-p1}
\Psi(H,H^2-4K)=-\lambda\lan X,N\ran.
\end{equation}
At umbilical points, i.e., $k_1=k_2,$
\[
\frac{\partial W}{\partial k_1}\cdot\frac{\partial W}{\partial k_2}=\left(\frac{\partial \Psi}{\partial x_1}\right)^2\geq 0.
\]
Thus, equation (\ref{F-1R}) is always weakly parabolic near umbilical points, and if we assume that $\Psi$ satisfies
\[
\frac{\partial \Psi}{\partial x_1}\neq 0,
\]
then equation (\ref{F-1R}) will be parabolic near umbilical points. 

\begin{remark}
{\normalfont
The same condition 
\[
\frac{\partial W}{\partial k_1}\cdot\frac{\partial W}{\partial k_1}>0\ (\geq 0)
\] 
which assures the (weak) parabolicity of (\ref{F-1}), also guarantees the (weak) ellipticity of (\ref{soliton}). Analogously, the condition (\ref{parabol-ellip-2}) assures the (weak) parabolicity of (\ref{F-1R}) and the (weak) ellipticity of (\ref{soliton-p1}).
}
\end{remark}

For each point $p\in\Sigma,$ we can take isothermal parameters $u$ and $v$ in a neighborhood of $p,$ i.e., 
\[
ds^2=\rho(u,v)(du^2+dv^2),
\]
where $ds^2$ is the metric of $\Sigma$ and $\rho$ is a positive smooth function on $\Sigma$. Complexifying the parameters by taking $z=u+iv,$ we can identify $\Sigma$ with a subset of $\mathbb{C}.$ In this case, we have
\[
\lan X_z,X_{\bar{z}}\ran = \frac{\rho(z)}{2} \ \mbox{and} \ ds^2=\rho(z)|dz|^2.
\]
Let us denote by
\[
P dz^2 = \lan B(X_z,X_z),N\ran dz^2
\]
the $(2,0)$-part of the second fundamental form of $\Sigma$ in $\R^3.$ Here $B(X,Y):=\nabla_XY-\nabla^{\Sigma}_XY.$ This quadratic form is also called the Hopf quadratic differential.

The immersion $X$ satisfies the equations
\begin{equation}\label{codazzi}
\left\{
\begin{aligned}
\n_{X_z}X_z&=\frac{\rho_z}{\rho}X_z + PN,\\
\n_{X_{\bar{z}}}X_z&=\frac{\rho}{4}HN,\\
\n_{X_{\bar{z}}}X_{\bar{z}}&= \frac{\rho_{\bar{z}}}{\rho}X_{\bar{z}} + \bar{P}N,\\
\end{aligned}
\right.
\end{equation}
and
\begin{equation}\label{codazzi-2}
\left\{
\begin{aligned}
\n_{X_z}N &=-\frac{1}{2} HX_z - \frac{2}{\rho}PX_{\bar{z}}\\
\n_{X_{\bar{z}}}N&= - \frac{2}{\rho}\bar{P}X_z - \frac{1}{2}HX_{\bar{z}}.\\
\end{aligned}
\right.
\end{equation}
Since 
\begin{equation}\label{hopf-diff-000}
\begin{aligned}
P&=\lan \n_{X_z}X_z,N\ran = \frac{1}{4}\lan \n_{X_u-iX_v}X_u-iX_v,N\ran\\
 &=\frac{1}{4}[\lan \n_{X_u}X_u,N\ran - \lan \n_{X_v}X_v,N\ran - i(\lan \n_{X_u}X_v,N\ran + \lan \n_{X_v}X_u,N\ran)]\\
 &=\frac{1}{4}[II(X_u,X_u) - II(X_v,X_v) - 2i II(X_u,X_v)],\\
\end{aligned}
\end{equation}
where $II$ is the second fundamental form of $\Sigma$ in $\R^3$, we have $P=0$ if and only if $II$ is umbilical. 

Let $e_1=(1/\sqrt{\rho})X_u$ and $e_2=(1/\sqrt{\rho})X_v$ be an orthonormal frame for $T\Sigma.$ Denoting by $h_{ij}=II(e_i,e_j),$ and using (\ref{hopf-diff-000}), gives
\begin{equation}\label{|P|}
\begin{aligned}
|P|^2&=\frac{1}{16}\left[(II(X_u,X_u) - II(X_v,X_v))^2 + 4II(X_u,X_v)^2\right]\\
	 &=\frac{\rho^2}{16}\left[(II(e_1,e_1) - II(e_2,e_2))^2 + 4II(e_1,e_2)^2\right]\\
     &=\frac{\rho^2}{16}\left[(h_{11}-h_{22})^2 + 4(h_{12})^2\right]\\	 
	 &=\frac{\rho^2}{16}[(h_{11}+h_{22})^2 - 4(h_{11}h_{22}-(h_{12})^2)]\\
	 &=\frac{\rho^2}{16}(H^2-4K).
\end{aligned}
\end{equation}

In order to prove main the theorem of this section, we will need the following result about complex functions which can be found in \cite{A-dC-T}, see also \cite{E-T-1} and \cite{E-T}.
\begin{lemma}\label{main}
Let $h:U\subset\CC\to\CC$ be a complex function defined in an open set $U$ of the complex plane. Assume that 
\begin{equation}\label{cauchy}
\left|\frac{\partial h}{\partial\bar{z}}\right|\leq \vp(z)|h(z)|,
\end{equation}
where $\vp$ is a continuous, non-negative real function. Assume further that $z=z_0\in U$ is a zero of $h.$ Then either $h=0$ in a neighbourhood $V\subset U$ of $z_0,$ or
\[
h(z)=(z-z_0)^k h_k(z_0), \ z\in V, \ k\geq1,
\]
where $h_k(z)$ is a continuous function with $h_k(z_0)\neq 0.$
\end{lemma}
\begin{remark}
{\normalfont
This result uses a weak notion of holomorphy to obtain that a complex function $h(z)$ with a zero $z_0$ is identically zero in a neighbourhood of $z_0$ or
\begin{equation}\label{eq-001}
h(z)=(z-z_0)^kh_k(z), \ z\in V, \ k\geq 1,
\end{equation}
for some complex function $h_k$ such that $h_k(z_0)\neq0.$ We remark that, when $h(z)$ is holomorphic, this is a well known property, see \cite{Rudin}, p. 208--209. This property will be crucial to prove our classification theorems since the Hopf quadratic differential is not necessarily holomorphic for surfaces satisfying (\ref{soliton}), p. \pageref{soliton}. The existence of a weak notion of holomorphy to conclude (\ref{eq-001}) was noticed first, as we know, by Carleman in 1933, see \cite{Carleman}. In fact, he proved that a solution $h:U\subset \CC \rightarrow\CC$ of 
\[
\frac{\partial h}{\partial\bar{z}} = ah+b\bar{h},
\]
does not admits a zero of infinite order except if $h=0.$ Notice that, if $a=b=0,$ then $h$ is holomorphic. Using these ideas, Hartman and Wintner, see \cite{H-W-2} and \cite{H-W}, and Chern, see \cite{chern}, proved their well known results on the classification of special Weingarten surfaces. 
}
\end{remark}

Now, we are ready to prove Theorem \ref{thm-soliton-0}, p. \pageref{thm-soliton-0}.

\begin{proof}
Let 
\[
h(z)=\lan\n_{X_z}X_z,N\ran=P(z),
\]
be the Hopf differential, where $N$ is the unitary normal vector field of the immersion. We have
\begin{equation}\label{Pz}
\begin{aligned}
\frac{\partial h}{\partial\bar{z}}&=\frac{\partial}{\partial \bar{z}}\lan \n_{X_z}X_z,N\ran\\
&= \lan\n_{X_{\bar{z}}}\n_{X_z}X_z,N\ran + \lan\n_{X_z}X_z,\n_{X_{\bar{z}}}N\ran\\
&=\lan R(X_z,X_{\bar{z}})X_z,N\ran + \lan\n_{X_z}\n_{X_{\bar{z}}}X_z,N\ran + \lan\n_{X_z}X_z,\n_{X_{\bar{z}}}N\ran\\
&=\frac{\partial}{\partial z}\left(\lan\n_{X_{\bar{z}}}X_z ,N\ran\right) - \lan\n_{X_{\bar{z}}}X_z,\n_{X_z}N\ran + \lan\n_{X_z}X_z,\n_{X_{\bar{z}}}N\ran\\
&= \frac{\partial}{\partial z}\left(\frac{\rho}{4}H\right)- \left\lan\frac{\rho}{4}HN, - \frac{1}{2}HX_z - \frac{2P}{\rho}X_{\bar{z}}\right\ran\\
&\qquad + \left\lan\frac{\rho_z}{\rho}X_z + PN,-\frac{2\bar{P}}{\rho}X_z - \frac{1}{2}H X_{\bar{z}}\right\ran\\
&=\frac{\rho}{4}H_z,
\end{aligned}
\end{equation}
where $R(X_z,X_{\bar{z}})X_z=0$ is the Euclidean curvature tensor. Differentiating (\ref{soliton-0}), p. \pageref{soliton-0}, with respect to $z$ gives 
\[
\Psi_1H_z + \Psi_2(H^2-4K)_z = -\lambda\lan X,N\ran_z
\]
where, here and after, $\Psi_1=\frac{\partial\Psi}{\partial x_1}$ and $\Psi_2=\frac{\partial\Psi}{\partial x_2}.$ By using (\ref{|P|}) and that $\Psi_1\neq0$, we have
\begin{equation}\label{Hz}
\begin{aligned}
H_z&=-\frac{\Psi_2}{\Psi_1}(H^2-4K)_z -\frac{\lambda}{\Psi_1}\lan X,N\ran_z\\
   &=-\frac{\Psi_2}{\Psi_1}(16\rho^{-2}|P|^2)_z -\frac{\lambda}{\Psi_1}\lan X_z,N\ran -\frac{\lambda}{\Psi_1}\lan X,\n_{X_z}N\ran\\
   &=32\frac{\Psi_2}{\Psi_1}\rho^{-3}\rho_z|P|^2 - 16\frac{\Psi_2}{\Psi_1}\rho^{-2}(\bar{P}P_z + P\bar{P}_z)\\
   &\qquad-\frac{\lambda}{\Psi_1}\left\lan X,-\frac{1}{2}HX_z - \frac{2}{\rho}PX_{\bar{z}}\right\ran\\
   &=32\frac{\Psi_2}{\Psi_1}\rho^{-3}\rho_z|P|^2 - 16\frac{\Psi_2}{\Psi_1}\rho^{-2}(\bar{P}P_z + P\bar{P}_z)\\
   &\qquad+\frac{\lambda}{2\Psi_1}H\lan X,X_z\ran + \frac{2\lambda}{\rho\Psi_1}P\lan X,X_{\bar{z}}\ran.\\  
\end{aligned}
\end{equation}
Replacing (\ref{Hz}) into (\ref{Pz}) and taking the modulus, we have
\[
\begin{aligned}
\left|\frac{\partial h}{\partial\bar{z}}\right|&\leq\left[8\left|\frac{\Psi_2}{\Psi_1}\right|\rho^{-2}|\rho_z||P| + 4\left|\frac{\Psi_2}{\Psi_1}\right|\rho^{-1}(|P_z|+|\bar{P}_z|)\right.\\
&\quad\left.+\frac{|\lambda|}{2|\Psi_1|}|\lan X,X_{\bar{z}}\ran|\right]|h(z)| + \frac{\rho|\lambda|}{8|\Psi_1|}|H||\lan X,X_z\ran|.\\
\end{aligned}
\]
Since
\[
X=\frac{2}{\rho}\lan X,X_{\bar{z}}\ran X_z + \frac{2}{\rho}\lan X,X_z\ran X_{\bar{z}} + \lan X,N\ran N,
\]
$\lan X_z,X_z\ran=0=\lan X_{\bar{z}},X_{\bar{z}}\ran,\ \lan X_z,X_{\bar{z}}\ran=\rho/2,$ and 
\[
\begin{aligned}
|\lan X,X_z\ran|&=\frac{1}{2}|\lan X,X_u\ran - i \lan X,X_v\ran|\\
&=\frac{1}{2}\sqrt{\lan X,X_u\ran^2 + \lan X,X_v\ran^2}\\
&=\frac{1}{2}|\lan X,X_u\ran + i \lan X,X_v\ran|\\
&=|\lan X,X_{\bar{z}}\ran|,
\end{aligned}
\] 
we have
\begin{equation}\label{Xtop}
\sqrt{\|X\|^2 - \lan X,N\ran^2}=\frac{2}{\sqrt{\rho}}\sqrt{|\lan X,X_z\ran||\lan X,X_{\bar{z}}\ran|}=\frac{2}{\sqrt{\rho}}|\lan X,X_z\ran|.
\end{equation}
This gives
\[
\begin{aligned}
\left|\frac{\partial h}{\partial\bar{z}}\right|&\leq \left[8\left|\frac{\Psi_2}{\Psi_1}\right|\rho^{-2}|\rho_z||P| + 4\left|\frac{\Psi_2}{\Psi_1}\right|\rho^{-1}(|P_z|+|\bar{P}_z|)\right.\\
&\left.\quad +\frac{|\lambda|\sqrt{\rho}}{4|\Psi_1|}\sqrt{\|X\|^2 - \lan X,N\ran^2}\right]|h(z)|\\
&\quad + \frac{\rho^{3/2}}{16|\Psi_1|}|\lambda||H|\sqrt{\|X\|^2-\lan X,N\ran^2}.\\
\end{aligned}
\]
On the other hand, the hypothesis (\ref{hyp-soliton-0}), p. \pageref{hyp-soliton-0},
\[
K\leq \frac{1}{4}[1-\varepsilon\lambda^2(\|X\|^2-\lan X,N\ran^2)]H^2,
\]
is equivalent to
\[
|H||\lambda|\sqrt{\|X\|^2-\lan X,N\ran^2}\leq \frac{1}{\sqrt{\varepsilon}}\sqrt{H^2-4K}.
\]
This implies
\begin{equation}\label{Pz0}
\begin{aligned}
\left|\frac{\partial h}{\partial\bar{z}}\right|&\leq \left[8\left|\frac{\Psi_2}{\Psi_1}\right|\rho^{-2}|\rho_z||P| + 4\left|\frac{\Psi_2}{\Psi_1}\right|\rho^{-1}(|P_z|+|\bar{P}_z|)\right.\\
&\left.\quad+\frac{|\lambda|\sqrt{\rho}}{4|\Psi_1|}\sqrt{\|X\|^2-\lan X,N\ran^2}\right]|h(z)| + \frac{\rho^{3/2}}{16|\Psi_1|} \frac{1}{\sqrt{\varepsilon}}\sqrt{H^2-4K}\\
&\leq \left[8\left|\frac{\Psi_2}{\Psi_1}\right|\rho^{-2}|\rho_z||P| + 4\left|\frac{\Psi_2}{\Psi_1}\right|\rho^{-1}(|P_z|+|\bar{P}_z|)\right.\\
&\left.\quad+\frac{|\lambda|\sqrt{\rho}}{4|\Psi_1|}\sqrt{\|X\|^2-\lan X,N\ran^2}\right]|h(z)| + \frac{\sqrt{\rho}}{4\sqrt{\varepsilon}|\Psi_1|}|h(z)|\\
&=\left[8\left|\frac{\Psi_2}{\Psi_1}\right||\rho_z|\rho^{-2}|P| + 4\left|\frac{\Psi_2}{\Psi_1}\right|\rho^{-1}(|P_z|+|\bar{P}_z|)\right.\\
&\left.\quad+\frac{|\lambda|\sqrt{\rho}}{4|\Psi_1|}\sqrt{\|X\|^2-\lan X,N\ran^2} + \frac{\sqrt{\rho}}{4\sqrt{\varepsilon}|\Psi_1|}\right]|h(z)|,\\
\end{aligned}
\end{equation}
for every $z\in\Sigma.$ Here we have used (\ref{|P|})  which implies 
\[
|h(z)|=\frac{\rho}{4}\sqrt{H^2-4K}.
\] 
Defining $\vp_0$ as 
\[
\begin{aligned}
\vp_0&=8\left|\frac{\Psi_2}{\Psi_1}\right||\rho_z|\rho^{-2}|P| + 4\left|\frac{\Psi_2}{\Psi_1}\right|\rho^{-1}(|P_z|+|\bar{P}_z|)\\
&\quad+\frac{|\lambda|\sqrt{\rho}}{4|\Psi_1|}\sqrt{\|X\|^2-\lan X,N\ran^2} + \frac{\sqrt{\rho}}{4\sqrt{\varepsilon}|\Psi_1|}\\
\end{aligned}
\]
we have
\[
\left|\frac{\partial h}{\partial\bar{z}}\right|\leq\vp_0(z)|h(z)|.
\]

Thus, by Lemma \ref{main}, either $h(z),$ and thus $P$ is identically zero in a neighborhood $V$ of a zero $z_0,$ or this zero is isolated and the index of a direction field determined by $\textrm{Im}[P dz^2]=0$ is $-k/2$ (hence negative). If, for some coordinate neighborhood $V$ of zero, $P=0$, this holds for the whole $\Sigma;$ otherwise, the zeroes on the boundary of $V$ will contradict Lemma \ref{main}. In fact, if $V\neq\Sigma,$ then its boundary $\partial V\neq \emptyset.$ Since the set of zeros is a closed set, the points of $\partial V$ are also zeros of $P.$ Since the points of $\partial V$ are not isolated, by Lemma \ref{main}, given $z_1\in\partial V,$ there exists another neighborhood $V_1\ni z_1,$ such that $P=0$ in $V_1,$ i.e. $z_1$ is an interior point of the set of zeros of $P.$ This contradiction implies that $V=\Sigma.$ On the other hand, if $h$ (and hence $P$) is not identically zero, all zeroes are isolated and have negative indices. Since $\Sigma$ has genus zero, by the Poincar\'e index theorem, the sum of the indices of the singularities of any field of directions is $2$ (hence positive). This
contradiction shows that $P$ is identically zero. Therefore, by (\ref{|P|}), we conclude that $\Sigma$ is totally umbilical, i.e., a round sphere.

If $\Sigma$ is a sphere and $\lambda\neq 0$, we now determine its radius $R$ and its center $x_0.$ First, notice that\begin{equation}\label{H-nu-2}
H=\frac{2}{R}\ \mbox{and} \ K=\frac{1}{R^2}.
\end{equation} 
Since $\Psi(H,H^2-4K)= -\lambda\lan X,N\ran,$ we have
\begin{equation}\label{X-nu}
\lan X,N\ran = -\lambda^{-1}\Psi\left(\frac{2}{R},0\right).
\end{equation}
On the other hand, $X(\Sigma)=\ss^{2}(x_0,R)$ implies $X=x_0-RN.$ This gives $\lan X,N\ran = \lan x_0,N\ran - R,$ i.e.,
\begin{equation}\label{x0-nu}
\lan x_0,N\ran = R -\lambda^{-1}\Psi\left(\frac{2}{R},0\right),
\end{equation}
which implies that $\lan x_0,N\ran$ is constant. Since $x_0$ is a fixed vector and $N$ varies over the entire $\ss^2(x_0,R),$ we conclude that $x_0=0,$ i.e., the sphere is centered at the origin. By using (\ref{x0-nu}), we have
\[
0 = R -\lambda^{-1}\Psi\left(\frac{2}{R},0\right),
\]
i.e.,
\[
\lambda R=\Psi\left(\frac{2}{R},0\right).
\]

In the case when $\lambda=0,$ then clearly the radius of the sphere satisfies $\Psi\left(\frac{2}{R},0\right)=0.$ Since the equation $\Psi(H,H^2-4K)=0$ is invariant by isometries of $\R^3,$ the center of the sphere can be any point of $\R^3.$ 

\end{proof}

Now we present the proofs of the corollaries presented in the introduction.

\begin{proof}[Proof of Corollary \ref{Hn-alpha}]
Just notice that, differentiating (\ref{soliton-Hn}) with respect to $z,$
\[
\beta H^{\beta-1}H_z=-\lambda\lan X,N\ran_z,
\]
i.e,
\[
H_z=-\lambda \beta^{-1}H^{1-\beta}\lan X,N\ran_z.
\]
Thus, if $\beta>1$ we will need to assume that $H\neq 0,$ but when $\beta=\frac{m}{2n-1}\leq 1, m,n\in\mathbb{N},$ $2n-1$ is odd number, this assumption  is not necessary. The conclusion of the proof then follows, step by step, the proof of Theorem \ref{thm-soliton-0}.
\end{proof}

\begin{proof}[Proof of Corollary \ref{harm-H}] Since by notations  $x_1=H$, $x_2=H^2-4K\ge 0$, then $4K=H^2-x_2$. Therefore, we consider $\Psi(x_1,x_2)=\left(\frac{x_1}{4}-\frac{x_2}{4x_1}\right)^{\frac{m}{2n-1}}.$  Since 
\[
\begin{aligned}
\frac{\partial\Psi}{\partial x_1}&=\frac{m}{2n-1}\left(\frac{x_1^2-x_2}{4x_1}\right)^{\frac{m}{2n-1}-1}\left(\frac{1}{4}+\frac{x_2}{4x_1^2}\right)\\
&=\frac{m}{2n-1}\left(\frac{4x_1}{x_1^2-x_2}\right)^{1-\frac{m}{2n-1}}\left(\frac{1}{4}+\frac{x_2}{4x_1^2}\right)\\
\end{aligned}
\] 
we have, for $\frac{m}{2n-1}\leq1,$
\[
\frac{\partial\Psi}{\partial x_1}\neq 0 \iff H\neq 0.
\]
Notice that $\frac{\partial\Psi}{\partial x_1}$ has a singularity at $K=0,$ but since we need only $\left(\frac{\partial\Psi}{\partial x_1}\right)^{-1}$ in the proof of Theorem \ref{thm-soliton-0}, this singularity will be removable.
\end{proof}

\begin{proof}[Proof of Corollary \ref{K-alpha}.] Notice that $4K=x_1^2-x_2$. We can take 
\[
\Psi(x_1,x_2)=\left(\frac{x_1^2-x_2}{4}\right)^{\alpha}
\]
in Theorem \ref{thm-soliton-0} and observe that 
\[
\frac{\partial\Psi}{\partial x_1}=\frac{\alpha x_1}{2}\left(\frac{x_1^2-x_2}{4}\right)^{\alpha-1}\neq 0,
\]
since $\Sigma$ is convex.
\end{proof}

\section{Examples}\label{example}

In this section we prove that hypothesis (\ref{hyp-soliton-0}), p. \pageref{hyp-soliton-0}, of Theorem \ref{thm-soliton-0} is necessary. Since the expression in (\ref{hyp-soliton-0})
\[
K\leq \dfrac{1}{4}[1-\varepsilon\lambda^2(\|X\|^2-\lan X,N\ran^2)]H^2
\]
is equivalent to
\[
\dfrac{|H|\sqrt{\|X\|^2-\lan X,N\ran^2}}{\sqrt{H^2-4K}}\leq \dfrac{1}{|\lambda|\sqrt{\ve}},
\]
the existence of a constant $\ve>0$ such that (\ref{hyp-soliton-0}) holds is equivalent to the boundedness of
\begin{equation}\label{quot}
\frac{|H|\sqrt{\|X\|^2-\lan X,N\ran^2}}{\sqrt{H^2-4K}}.
\end{equation}
Notice that the only possibility for (\ref{quot}) to be unbounded is at the umbilical points, provided $H^2-4K=0$ in these points. 

We know that, if a rotational surface $\Sigma$ intersects the axis of rotation orthogonally, then this intersection point is umbilical. If $\Sigma$ is a rotational surface of genus zero, then there are two of these points. The aim of this section is to  show that, if there exists any non spherical rotational surface satisfying
\[
\Psi(H,H^2-4K)=-\lambda\lan X,N\ran
\]
which intersects the rotation axis orthogonally, then (\ref{quot}) is unbounded in the neighborhood of these umbilical points. Here we will assume that $\Psi$ is homogeneous of degree $\beta\in\R,$ i.e., $\Psi(a x_1,a^2 x_2)=a^\beta\Psi(x_1,x_2),$ $a>0,$ and $\frac{\partial \Psi}{\partial x_1}\neq 0.$ 

\begin{figure}[ht]
\includegraphics[scale=0.5]{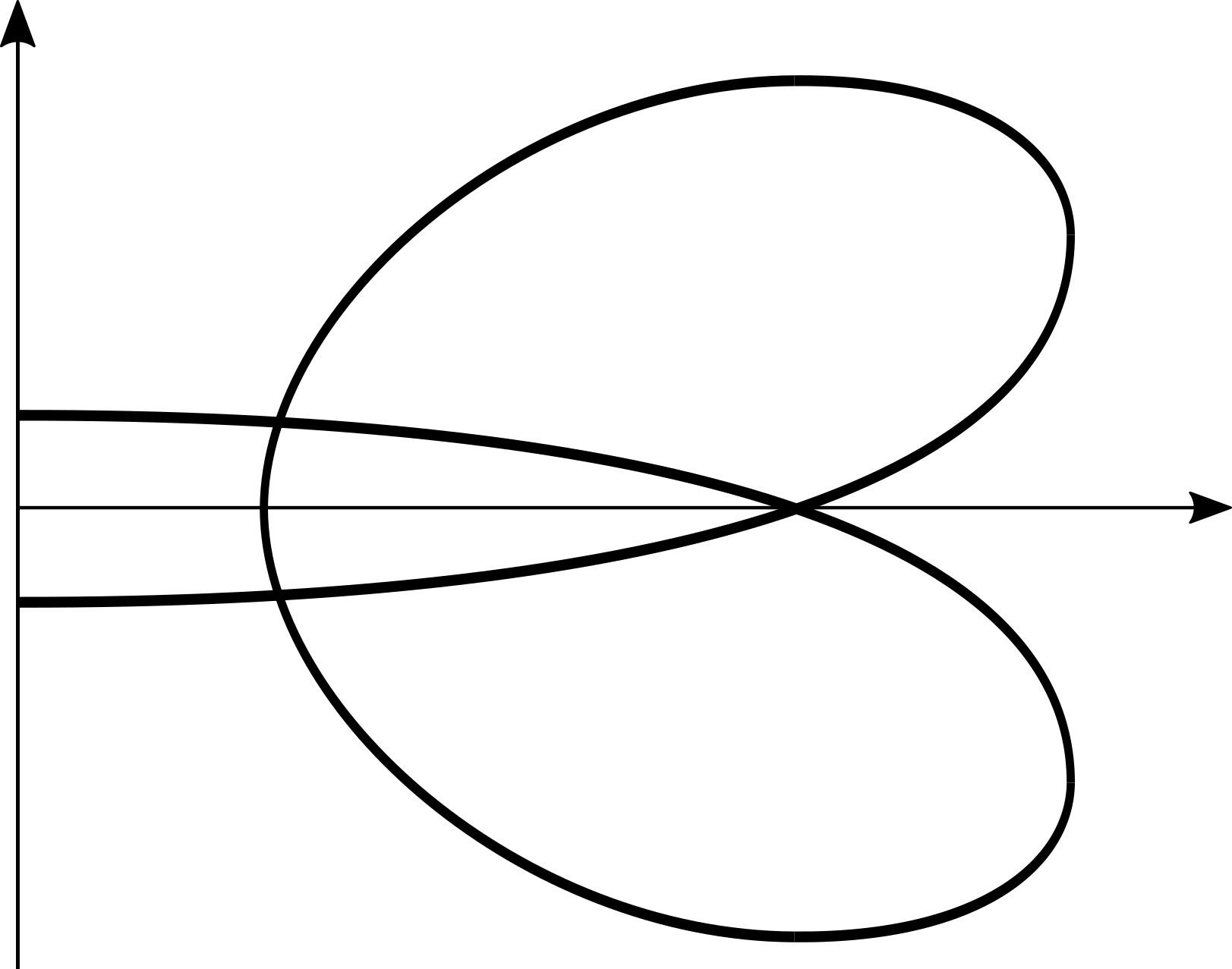}
\caption{Draft of the profile curve of the Drugan's genus zero self-shrinker of the mean curvature flow, see \cite{D}. The surface is obtained by rotating the profile curve around the vertical axis. The intersection of the profile curve with the rotation axis gives two isolated umbilical points which do not satisfy the hypothesis (\ref{eq-Hss}) of Corollary \ref{H-ss}}
\label{aaaaa}
\end{figure}

In order to prove our claim let us recall some properties of rotational surfaces in $\R^3.$ Let 
\[
X(t,\theta)=(x(t)\cos\theta,x(t)\sin\theta,y(t))
\]
be a parametrization of a rotational surface, where $t\in[0,\eta),$ $\eta>0,$ and $\theta\in\R.$ The principal curvatures of $\Sigma=X([0,\eta)\times\R)$ are given by 
\[
k_1(t)=\frac{-y'(t)}{x(t)\sqrt{(x'(t))^2+(y'(t))^2}}\ \mbox{and}\ k_2(t)=\frac{x''(t)y'(t)-x'(t)y''(t)}{((x'(t))^2+(y'(t))^2)^{3/2}}.
\]
If the profile curve is a graph $(x,\gamma(x))$, then we have
\[
k_1(x)=\frac{-\gamma'(x)}{x\sqrt{1+(\gamma'(x))^2}}\ \mbox{and} \ k_2(x)=\frac{-\gamma''(x)}{(1+(\gamma'(x))^2)^{3/2}}.
\]
From now on, we will consider that the profile curve of $\Sigma$ is locally a graph $(x,\gamma(x))$ near the umbilical point $x=0.$ Next, we will find the expressions of (\ref{quot}) in therms of $\gamma.$ First, we have
\[
H=k_1+k_2 = -\frac{1}{(1+(\gamma'(x))^2)^{3/2}}\left[\frac{\gamma'(x)}{x}(1+(\gamma'(x))^2) + \gamma''(x)\right]
\]
and
\[
H^2-4K=(k_1-k_2)^2=\frac{1}{(1+(\gamma'(x))^2)^3}\left[\frac{\gamma'(x)}{x}(1+(\gamma'(x))^2) - \gamma''(x)\right]^2.
\]
Since the inward unit normal $N$ of $\Sigma$ is given by
\[
N=\frac{1}{\sqrt{1+(\gamma'(x))^2}}(\gamma'(x)\cos\theta,\gamma'(x)\sin\theta,-1),
\]
we have
\[
\begin{aligned}
\|X\|^2-\lan X,N\ran^2& = x^2+(\gamma(x))^2 - \frac{(x\gamma'(x)-\gamma(x))^2}{1+(\gamma'(x))^2}\\
&=\frac{(x+\gamma(x)\gamma'(x))^2}{1+(\gamma'(x))^2}.\\
\end{aligned}
\]
Therefore
\begin{equation}\label{quot-gamma}
\frac{|H|\sqrt{\|X\|^2-\lan X,N\ran^2}}{\sqrt{H^2-4K}}=\frac{\left|\frac{\gamma'(x)}{x}(1+(\gamma'(x))^2) + \gamma''(x)\right||x+\gamma(x)\gamma'(x)|}{\left|\frac{\gamma'(x)}{x}(1+(\gamma'(x))^2) - \gamma''(x)\right|\sqrt{1+(\gamma'(x))^2}}.
\end{equation}

In order to conclude our analysis, we will need the Taylor expansion of $\gamma$ satisfying the initial condition $\gamma(0)=b>0$ and $\gamma'(0)=0.$ This second condition means that $\Sigma$ intersects the rotation axis (the $z$ axis) orthogonally.

In our situation, equation (\ref{soliton-0}), p. \pageref{soliton-0}, becomes
\[
\begin{aligned}
&\Psi\left(-\frac{\left[\frac{\gamma'(x)}{x}(1+(\gamma'(x))^2) + \gamma''(x)\right]}{(1+(\gamma'(x))^2)^{3/2}},\frac{\left[\frac{\gamma'(x)}{x}(1+(\gamma'(x))^2) - \gamma''(x)\right]^2}{(1+(\gamma'(x))^2)^3}\right)\\
&=-\lambda\frac{x\gamma'(x)-\gamma(x)}{\sqrt{1+(\gamma'(x))^2}}.
\end{aligned}
\]
Assuming that $\Psi$ is homogeneous of degree $\beta,$ i.e., \[\Psi(a x_1,a^2 x_2)=a^\beta\Psi(x_1,x_2),\] we have
\begin{equation}\label{soliton-rot}
\begin{aligned}
&\frac{\Psi\left(-\frac{\gamma'(x)}{x}(1+(\gamma'(x))^2) - \gamma''(x),\left[\frac{\gamma'(x)}{x}(1+(\gamma'(x))^2) - \gamma''(x)\right]^2\right)}{(1+(\gamma'(x))^2)^{3\beta/2}}\\
&=-\lambda\frac{x\gamma'(x)-\gamma(x)}{\sqrt{1+(\gamma'(x))^2}}.
\end{aligned}
\end{equation}
Notice that, since $\lim_{x\to 0^+}\frac{\gamma'(x)}{x}=\gamma''(0),$ equation (\ref{soliton-rot}) is not singular at $x=0.$
\begin{proposition}
Let $\gamma(x)$ be the solution of (\ref{soliton-rot}) with the initial conditions $\gamma(0)=b>0$ and $\gamma'(0)=0,$ where we assume that $\Psi$ is homogeneous of degree $\beta$ and $\frac{\partial \Psi}{\partial x_1}\neq 0$ in a neighborhood of $x=0.$ Then, near $x=0,$ we have
\begin{equation}\label{gamma}
\gamma(x)=b+cx^2+O(x^4),
\end{equation}
where $c=-\frac{1}{4}\left(\frac{\lambda b}{\Psi(1,0)}\right)^{1/\beta}.$
\end{proposition}
\begin{proof}
First, taking $x\to0^+$ in (\ref{soliton-rot}), we obtain
\[
\Psi(-2\gamma''(0),0)=-\lambda b\ \implies \ (-2\gamma''(0))^\beta\Psi(1,0)=\lambda b,
\]
i.e.,
\begin{equation}\label{c}
\gamma''(0)=-\frac{1}{2}\left(\frac{\lambda b}{\Psi(1,0)}\right)^{1/\beta}:=2c.
\end{equation}
Let
\[
\gamma(x)=b+cx^2 + a_3x^3 + a_4x^4 + O(x^5)
\]
be the Taylor expansion of $\gamma(x)$ near $x=0.$ Therefore, we have
\[
\gamma'(x)=2cx + 3a_3x^2 + 4a_4x^3 + O(x^4)
\]
and 
\[
\gamma''(x)=2c + 6a_3x + 12a_4x^2 + O(x^3).
\]
This implies
\begin{equation}\label{g-1}
\begin{aligned}
-\frac{\gamma'(x)}{x}(1+(\gamma'(x))^2) - \gamma''(x)&=-(2c+3a_3x+4a_4x^2+O(x^3))\times\\
&\quad\times(1+4c^2x^2+O(x^3))\\
&\quad - (2c+6a_3x+12a_4x^2+O(x^3))\\
&=-4c-9a_3x -8(c^3+2a_4)x^2 + O(x^3),\\
\end{aligned}
\end{equation}
i.e.,
\begin{equation}\label{g-2}
\begin{aligned}
\frac{\gamma'(x)}{x}(1+(\gamma'(x))^2) - \gamma''(x)&= -3a_3x+8(c^3-a_4)x^2 + O(x^3),\\
\end{aligned}
\end{equation}
and
\begin{equation}\label{g-3}
\begin{aligned}
x\gamma'(x)-\gamma(x)&=x(2cx+3a_3x^2+O(x^3)) -(b+cx^2+O(x^3))\\
&=-b + cx^2 + O(x^3).
\end{aligned}
\end{equation}
Replacing (\ref{g-2}) and (\ref{g-3}) into (\ref{soliton-rot}), gives
\begin{equation}\label{sol-gamma-1}
\begin{aligned}
\Psi&\left(-4c-9a_3x -8(c^3+2a_4)x^2 + O(x^3),[-3a_3x+8(c^3-a_4)x^2 + O(x^3)]^2\right)\\
&=-\lambda(1+4c^2x^2+O(x^3))^{\frac{3\beta-1}{2}}(-b + cx^2 + O(x^3))\\
\end{aligned}
\end{equation}
Taking the derivative in (\ref{sol-gamma-1}),
\[
\begin{aligned}
&-\frac{\partial\Psi}{\partial x_1}\cdot(9a_3+16(c^3+2a_4)x+O(x^2))+\frac{\partial\Psi}{\partial x_2}\cdot(18a_3^2x +O(x^2))\\
&=-\lambda\left(\frac{3\beta-1}{2}\right)(1+4c^2x^2+O(x^3))^{\frac{3}{2}(\beta-1)}\times\\
&\quad\times(8c^2x+O(x^2))(-b + cx^2 + O(x^3))\\
&\quad -\lambda (1+4c^2x^2+O(x^3))^{\frac{3\beta-1}{2}}(2cx+O(x^2)).\\
\end{aligned}
\]
implies that, at $x=0,$
\[
-9a_3\frac{\partial\Psi}{\partial x_1}=0 \implies a_3=0,
\]
since $\frac{\partial\Psi}{\partial x_1}\neq 0$ at $x=0$ by hypothesis. Thus
\begin{equation}\label{gamma-2}
\gamma(x)=b+cx^2 + a_4x^4 + O(x^5),
\end{equation}
where $c$ is given by (\ref{c}).
\end{proof}
The main result of the this section is the following
\begin{proposition}
Let $\Sigma$ be a (piece of a) rotational surface of $\R^3$ satisfying 
\[
\Psi(H,H^2-4K)=-\lambda\lan X,N\ran,\ \lambda\in\R\backslash\{0\},
\]
where $\Psi$ satisfies $\Psi(a x_1,a^2 x_2)=a^\beta\Psi(x_1,x_2),\ a>0,\ \beta\in\R,$ and $\frac{\partial\Psi}{\partial x_1}\neq 0.$ If $\Sigma$ intersects the axis of rotation orthogonally and is not a (piece of a) sphere centered at the origin, then
\[
\frac{|H|\sqrt{\|X\|^2-\lan X,N\ran^2}}{\sqrt{H^2-4K}}
\]
is unbounded in a neighborhood of the umbilical point which intersects the axis of rotation.
\end{proposition}
\begin{proof}
By replacing (\ref{gamma}) in (\ref{quot-gamma}) (in order to simplify the calculations we can use (\ref{g-1}) and (\ref{g-2}) for $a_3=0$), we have
\[
\begin{aligned}
&\frac{|H|\sqrt{\|X\|^2-\lan X,N\ran^2}}{\sqrt{H^2-4K}}\\
&=\frac{|4c+8(c^3+2a_4)x^2+O(x^3)||x+(b+cx^2+O(x^4))(2cx+O(x^3))|}{|8(c^3-a_4)x^2+O(x^3)|\sqrt{1+4c^2x^2+O(x^3)}}\\
&=\frac{|4c+8(c^3+2a_4)x^2+O(x^3)||1+2bc + O(x^2)|}{x|8(c^3-a_4)+O(x)|\sqrt{1+4c^2x^2+O(x^3)}}\\
&:=\frac{\tilde{F}(x)}{x},
\end{aligned}
\]
where $\tilde{F}(0)=\frac{c(1+2bc)}{2(c^3-a_4)}\neq 0$ if and only if $1+2bc\neq 0.$ But
\[
1+2bc=0 \iff 1=\frac{b}{2}\left(\frac{\lambda b}{\Psi(1,0)}\right)^{1/\beta},
\]
i.e.,
\[
b^{1+\frac{1}{\beta}}=\frac{2\Psi(1,0)^{\frac{1}{\beta}}}{\lambda^\frac{1}{\beta}}.
\]
On the other hand, the sphere centered at the origin and radius $b$ satisfies
\[
\begin{aligned}
\Psi\left(\frac{2}{b},0\right)=\lambda b &\iff \left(\frac{2}{b}\right)^\beta\Psi(1,0)=\lambda b \\
&\iff b^{1+\beta}=\frac{2^\beta}{\lambda}\Psi(1,0)\\
&\iff b^{1+\frac{1}{\beta}}=\frac{2\Psi(1,0)^{\frac{1}{\beta}}}{\lambda^\frac{1}{\beta}}.
\end{aligned}
\]
Thus, $\Sigma$ and $\ss^2(b)$ has a commom point at $x=0$. Since the profile curve of both surfaces has a horizontal tangent at $x=0,$ by the uniqueness theorem for ordinary differential equations, we conclude that $\Sigma=\ss^2(b)$ in a neighborhood of $x=0.$  Therefore, if $\Sigma$ is not a sphere, then $\tilde{F}(0)\neq 0.$ Thus, near zero,
\[
\frac{|H|\sqrt{\|X\|^2-\lan X,N\ran^2}}{\sqrt{H^2-4K}}=\frac{\tilde{F}(x)}{x}
\]
is unbounded.
\end{proof}

\end{document}